\documentclass[11pt, oneside]{article}   	
\usepackage{geometry}                		
\usepackage{graphicx}				
\usepackage{amssymb}
\usepackage{amsmath}
\usepackage{amsthm}
\usepackage{verbatim}
\usepackage{harpoon}
\usepackage{accents}
\newcommand{\vect}[1]{\accentset{\rightharpoonup}{#1}}
\usepackage{tikz}  
\usepackage{mathrsfs}
\usepackage{verbatim}

\usetikzlibrary{positioning,chains,fit,shapes,calc}  

\newtheorem{theorem}{Theorem}
\newtheorem{lemma}{Lemma}
\newtheorem{definition}{Definition}

\title{On the Jacobian Conjecture and ideal membership for degree $d$-linear maps}
\author{Mario DeFranco}

\begin{document} 

\maketitle

\abstract{We consider polynomial maps, which we call degree $d$-linear maps, that satisfy the Jacobian condition. We prove that certain infinite families of elements, which appear in the coefficients of the formal inverse of such maps, are in the ideal determined by the Jacobian condition. Using the Cayley-Hamilton theorem, we provide expressions for these elements in terms of the generators of that ideal. We also give a combinatorial proof of the Cayley-Hamilton theorem similar to that of Straubing \cite{Straubing}. We also include results of Gr\"obner basis computations regarding other elements.}

\section{Introduction}\label{intro}
The Jacobian Conjecture is a problem in algebraic geometry first stated by Ott-Heinrich Keller in 1939 (see \cite{Keller} and van den Essen \cite{van den Essen}). It has several formulations, and in this paper we consider the following.

For an integer $n \geq 1$, let $x = (x_1, \dots, x_n)\in \mathbb{C}^n$, define a polynomial map to be a map of the form  
\begin{align*}
&f \colon \mathbb{C}^n \rightarrow \mathbb{C}^n\\ 
&(x_1, \ldots, x_n) \mapsto (f_1(x_1, \ldots, x_n), \ldots, f_n(x_1, \ldots, x_n)  )
\end{align*}
where each component function $f_i(x_1, \ldots, x_n)$ is a polynomials function.

\textbf{Jacobian Conjecture}: Let $n \geq 2$. Suppose $f(x)$ is a polynomial map such that 
\begin{equation}\label{Jac condition}
\mathrm{Jac}(f)(x) := \det(D(f)(x)) = 1 
\end{equation}
for all $x \in \mathbb{C}^n$, where $\mathrm{Jac}(f)$ is the Jacobian of $f(x)$ and $D(f)(x)$ is the differential of $f(x)$. Then $f(x)$ has an inverse function that is also a polynomial map.  

Polynomial automorphisms appear in many topics in algebraic geometry and dynamical systems (see van den Essen \cite{van den Essen}). The Jacobian conjecture itself is greatly generalized by the Generalized Vanishing Conjecture (see van den Essen and Zhao \cite{van den Essen 08}). A connection has also been proposed to quantum field theory (see Abdesselam \cite{Abdesselam}). 

One approach to the Jacobian Conjecure is to consider the formal inverse of $f(x)$ (see Bass, Connell, Wright \cite{Bass}, Johnston \cite{Johnston 21}, Singer \cite{Singer}, Wright \cite{Wright}). This is \textit{a priori} an infinite power series in the indeterminates $x_i$ which may not converge on all of $\mathbb{C}^n$. If the condition \eqref{Jac condition} implies that  the coefficients of the monomials of the $x_i$ were zero for all but finitely many monomials, then the formal series would truncate to a polynomial, which would be the inverse function of $f(x)$, convergent everywhere. The Jacobian conjecture thus claims that the vanishing of the non-constant coefficients of the Jacobian \eqref{Jac condition} implies the vanishing of almost all of the coefficients of the formal inverse. 

In this paper we analyze the formal inverse of certain polynomial maps we call \newline degree $d$-linear maps:
\begin{definition}
For integers $d,n\geq1$, call a map $f$ 
\[
f \colon \mathbb{C}^n \rightarrow \mathbb{C}^n
\]
a degree $d$-linear map if the $i$-th component function is of the form 
\[
f_i(x_1, \ldots, x_n) = x_i- (\sum_{j=1}^n a_{i,j}x_r)^d
\]
for $a_{i,j}\in \mathbb{C}$. 
\end{definition}
To prove the Jacobian conjecture for all $n$, it is sufficient to prove it for these maps for all $n$ when $d=3$ (Dru\.zkowski \cite{Druzkowski 83}). In that case these maps are called cubic-linear.  

Now we describe the layout of this paper. We treat the $a_{i,j}$ as indeterminates that generate a ring $R$. The non-constant coefficients of $\mathrm{Jac}(f)(x)$ then generate an ideal $J_{d,n}$ of $R$. In Section \ref{formal inverse} we express each coefficient of the formal inverse as a sum of elements of $R$ over rooted plane trees. In Lemma \ref{l fern factorization}, we show that portions of these sums factorize using certain elements of $R$ which we denote by $\mathrm{fern}_{d,n}$, taking this terminology from Johnston and Prochno \cite{Johnston 19}. In Theorem \ref{t membership} we prove that certain infinite families of these elements are in $J_{d,n}$, and we express them in the basis of the generators of that ideal. We prove this formula by showing it is an instance of the Cayley-Hamilton theorem applied to a certain matrix. We note that Abdesselam \cite{Abdesselam} interprets the ``$d$=1'' Jacobian conjecture as the Cayley-Hamilton theorem. Thus Theorem \ref{t membership} generalizes that result to all $d$. We give a combinatorial proof of the Cayley-Hamilton theorem in Section \ref{CH} and note that it is very similar to that of Straubing \cite{Straubing}. Finally in Section \ref{groebner} we list some computations done with Groebner bases to check more ideal membership questions. We specify certain elements that are not in $J_{d,n}$ but are in its radical, and also consider further nilpotent assumptions of Dru\.zkowski \cite{Druzkowski 01}.  

\section{Formal inverse} \label{formal inverse}
We express the formal inverse of a degree $d$-linear map using rooted plane trees \newline (Lemma \ref{g form}). First we present definitions. 

\begin{definition} 
Let $a_{i,j}, 1 \leq i,j \leq n$ be indeterminates that generate a polynomial algebra $R$ over $\mathbb{Q}$. $R$ is a graded ring by total degree of a monomial. Let $A$ denote the $n\times n$ matrix 
\[
A_{i,j} = a_{i,j}.
\]
\end{definition}
Monomials in the ring $R$ will be used as weights for rooted plane trees.

\begin{definition} 
A rooted plane tree $T$ is an acyclic graph with a distinguished vertex called the root such that the children of each vertex are ordered, where a vertex $u$ of $T$ is a child of $v$ if $u$ is adjacent to $v$ and $v$ is on the path from $u$ to the root. Let $V(T)$ denote the vertex set of $T$. 

Define a labeled tree to be a pair $(T,\lambda)$ where $T$ is rooted plane tree and $\lambda$ is a function 
\[
\lambda \colon V(T) \rightarrow [1,n] 
\] from the vertex set of $T$ to the integers in $[1,n]$. Define a root-leaf labeled tree to be a pair $(T, \mu)$ where $T$ is a rooted plane tree and $\mu$ is a function 
\[
\mu \colon \{ \text{root and leaves of $T$} \} \rightarrow [1,n]. 
\]
We say that a labeled tree agrees with a root-leaf labeled tree if they have the same rooted plane tree and $\lambda = \mu$ when restricted to the root and leaves. 
Define $L(T,\mu)$ to the set of all labeled trees that agree with the root-leaf labeled tree $(T,\mu)$. 
 For a labeled tree $(T,\lambda)$ and an edge $e$ of $T$ going from parent vertex $v$ to child vertex $u$, define $w(e,\lambda)\in R$ to be
\[
w(e,\lambda) = a_{\lambda(v),\lambda(u)}.
\]
Define 
\[
w(T,\lambda)  = \prod_{e \in T} w(e, \lambda),
\]
and for a root-leaf labeled tree $(T,\mu)$, define
\[
z(T, \mu) = \sum_{(T,\lambda) \in L(T,\mu)} w(T,\lambda)
\]

\end{definition}



\begin{lemma} \label{l g form}
Suppose $f(x)$ is a degree $d$-linear map.
Let $g(x)$ denote the formal inverse of $f$, so as formal power series
\[
f_i( g_1(x_1, \ldots, x_n), \ldots, g_n(x_1, \ldots, x_n)  ) = x_i.
\] 
Then
\begin{equation}\label{g form}
g_i(x_1, \ldots, x_n) = \sum_{(T,\lambda),\, \lambda(\mathrm{root}(T)) = i} w(T) \prod_{v \in \mathrm{Leaves}(T)}  x_{\lambda(v)}
\end{equation}
where the sum is over all labeled trees with root labeled $i$ and such that every vertex of $T$ has degree $d$ or 0. 
\end{lemma}
\begin{proof}
 Let $\tilde{g}(x_1, \ldots, x_n)$ denote the right side of \eqref{g form}. The term $x_i$ appears in the series for $\tilde{g}(x_1, \ldots, x_n)$ because this term corresponds to tree with one vertex, its root, labeled $i$, which is also a leaf. Any other tree $T$ in this series has root of degree $d$ and is thus determined by a $d$-tuple of the subtrees whose root is a child of the root of $T$. Therefore the sum of all these other terms in this series has the factored form 
 \[
(\sum_{j=1}^n a_{i,j}\tilde{g}_j(x_1, \ldots, x_n))^d,
 \]
and so 
\[
f_i(\tilde{g}(x_1, \ldots, x_n)) = x_i.
\] 
 This completes the proof.  
\end{proof}

\begin{definition}
Let the height of a tree $T$ denote the maximal number of edges in a path from the root of $T$ to one of its leaf vertices. Denote the height by $h(T)$. If $h(t)\geq n$, then define the leftmost $n$-path $LP_n(T)$ of $T$ to be the leftmost path front the root that consists of $n$ edges. Given a labeling $\lambda$ of $T$, let $K(T,\lambda)$ denote the set of all labelings $\lambda'$ such that $\lambda'$ agrees with $\lambda$ on all vertices except possibly on those in the interior of $LP_n(T)$. \end{definition} 

We will consider certain trees we denote $\mathrm{fern}_{d,n}$ using the terminology of Johnston and Prochno \cite{Johnston 19}. \begin{definition} 
Define $\mathrm{fern}_{d,n}$ to be a rooted plane tree that consists of $n$ vertices of degree $d$ or 0 such that each vertex of degree $d$ has at most one child of degree $d$, and if such a child exists it is the leftmost of all the children of its parent.
\end{definition} 

We therefore obtain a partition of the set of labelings of $T$ into equivalence classes with each class consisting of those labelings that agree on all vertices except possibly those on the interior of $LP_n(T)$. 
\begin{definition}
We say a rooted plane tree $T$ is $d$-regular if its vertices have degree either $d$ or 0. For a rooted plane tree $T$ of height at least $n$, each of the above mentioned equivalence classes determines a labeling $\lambda_{\mathrm{fern}}$ of the tree $\mathrm{fern}_{d,n}$: the leaf vertices of $\mathrm{fern}_{d,n}$ correspond to the vertices that are the children of each vertex in $LP_n(T)$.
\end{definition}

 
\begin{lemma}\label{l fern factorization} Let $T$ be a rooted plane tree with height at least $n$ and labeling $\lambda$. Then
\[
\sum_{\eta \in K(T,\lambda)} w(T,\eta) = z(\mathrm{fern}_{d,n}, \lambda_{\mathrm{fern}})g  
\]
for some $g \in R$ that depends on $\lambda$. 
\end{lemma} 
\begin{proof} 
This follows immediately from the construction. The element $g$ is 
\[
g = \prod_{e} w(e, \lambda)
\]
where the product is over all edges of $T$ that are not adjacent to a vertex in the interior of $LP_n(T)$ or to the root of $T$. 
\end{proof} 


\section{Ideal membership}\label{ideal membership}

In this section we prove Theorem \ref{t membership} that certain elements are in the ideal determined by the Jacobian condition \eqref{Jac condition}.

\begin{definition} Let $\mathrm{Jac}(f)$ denote the Jacobian of $f$: 
\[
\mathrm{Jac}(f) = \det(D(f))
\]
where
\[
(D(f))_{i,j} = \frac{d}{dx_j}( f_i(x)). 
\]
Now $\mathrm{Jac}(f) \in R[x_1, \ldots, x_n]$. For a mult-degree $\alpha$, we let $J_\alpha \in R$ be the coefficient of $x^\alpha$:
\[
\mathrm{Jac}(f) = \sum_{\alpha} J_\alpha x^\alpha.
\]

Suppose $d \geq 2$. Let $J$ denote the ideal generated by the $J_\alpha$ for all $\alpha \neq \vect{0}$, i.e. by all coefficients of $\mathrm{Jac}(f)$ except for the constant term. 

If $d=1$, then $\mathrm{Jac}(f) \in R$, and we let $J$ be the ideal generated by all homogeneous components of $\mathrm{Jac}(f)$ except for the constant term. 
\end{definition}

\begin{lemma} \label{l J alpha form}
For integers $1\leq k,l\leq n$, let $\alpha(k,l) = (\alpha_1, \ldots, \alpha_n)$ be a multi-degree such that $\alpha_l=k(d-1)$ for some $1\leq k \leq n-1$, and $\alpha_i=0$ for all other $i$. Then 
\[
J_{\alpha(k,l)} = (-d)^k \sum_{(i_1, \ldots, i_k)} |A|_{(i_1, \ldots, i_k)} \prod_{r=1}^k a_{i_r,l}^{d-1}
\] 
where the sum is over all $k$-tuples $(i_1, \ldots, i_k)$ of distinct integers in $[1,n]$ with $i_r< i_{r+1}$, and where $|A|_{(i_1, \ldots, i_k)}$ denotes the principal minor of $A$ corresponding to rows $i_1, \ldots, i_k$. 
\end{lemma}
\begin{proof}
We have 
\begin{equation} \nonumber
(D(f))_{i,j} = \delta_{i,j} - da_{i,j}(\sum_{r=1}^n a_{i,r}x_r)^{d-1}
\end{equation}
where $\delta_{i,j}$ is the Kronecker delta. Apart from $\delta_{i,j}$, each entry of $D(f)$ is thus homogeneous of degree $d-1$ in the variables $x_i$, and in the expansion of the determinant each term of $x^{k(d-1)}$ is obtained by choosing 
\begin{equation}\label{x choice}
-da_{i,j}(\sum_{r=1}^n a_{i,r}x_r)^{d-1}
\end{equation}
$k$ times and choosing 1 $n-k$ times in the main diagonal. The $k$ rows $i_1, \ldots, i_k$ from which a term \eqref{x choice} was chosen determine the rows of a principal minor; the product of the coefficients $a_{i,j}$ outside the sum at \eqref{x choice} summed over those rows yield the principal minor $|A|_{(i_1, \ldots, i_k)}$, and taking the coefficient of $x_l$ from \eqref{x choice} yields the product $\displaystyle \prod_{r=1}^k a_{i_r,l}^{d-1}$. This completes the proof. 
\end{proof}


\begin{definition}For integers $i,j,l \in [1,n]$ not necessarily distinct, let $\mu(i,j,l)$ denote the root-leaf labeling of $\mathrm{fern}_{d,n}$ such that the root is labeled $i$, the vertex $v_n$ (the last vertex in the leftmost path) is labeled $j$, and all other vertices are labeled $l$. 
\end{definition}

\begin{lemma}\label{z matrix entry}For integers $i,j,l \in [1,n]$,
\[
z(\mathrm{fern}_{d,n}, \mu(i,j,l)) = (B^n)_{i,j}
\]
where $B$ is the $\times n$ matrix 
\[
B_{i,j} = a_{i,j} a_{i,l}^{d-1}.
\]
\end{lemma}
\begin{proof}
By construction 
\[
z(\mathrm{fern}_{d,n}, \mu(i,j,l)) = \sum_{\lambda} \prod_{k=0}^{n-1}(a_{\lambda_k, \lambda_{k+1}} a_{k,l}^{d-1})
\]
where the sum is over all $(n+1)$-tuples $\lambda$ 
\[
\lambda = (\lambda_0, \ldots, \lambda_n)
\]
where $\lambda_0=i$ and $\lambda_n=j$. But this is $(B^n)_{i,j}$ by definition. This completes the proof. 
\end{proof}

\begin{theorem}\label{t membership} For integers $i,j,l \in [1,n]$,
\[
z(\mathrm{fern}_{d,n}, \mu(i,j,l)) \in J_{d,n}
\]
and 
\[
z(\mathrm{fern}_{d,n},\mu(i,j,l) = -\sum_{k=0}^{n-1} (A^k)_{i,j}d^{-(n-k)} J_{\alpha(n-k,l)}. 
\]
\end{theorem}
\begin{proof}
Let $p_B(t)$ denote the characteristic polynomial of the matrix $B$ as in Lemma \ref{z matrix entry}: 
\[
p_B(t) = \det ( t I - B) = \sum_{k=0}^n c_k t^k
\]
where $c_k \in R$ and $c_n=1$. We claim that $c_k \in J_{d,n}, 0 \leq k \leq n-1$.

By expanding the determinant in the definition of the characteristic polynomial we have
\[
c_k = (-1)^{n-k} \sum_{(i_1, \ldots, i_{n-k})} |A|_{(i_1, \ldots, i_{n-k})} \prod_{r=1}^{n-k}k a_{i_r,l}^{d-1}
\]
with notation as in Lemma \ref{l J alpha form}. By that lemma 
we thus have 
\[
c_k =  d^{-(n-k)} J_{\alpha(n-k,l)}.
\]
Thus $c_k \in J_{d,n}$ for $0\leq k \leq n-1$.

Now by the Cayley-Hamilton theorem $p_B(B)= 0$, so
\[
\sum_{k=0}^n (B^k)_{i,j}c_k = 0.
\]
Since $c_n=1$, by Lemma \ref{z matrix entry} the above equation becomes
\[
z(\mathrm{fern}_{d,n},\mu(i,j,l) = -\sum_{k=0}^{n-1} (B^k)_{i,j}c_k. 
\]

 This completes the proof.
\end{proof}

\section{Combinatorial proof of Cayley-Hamilton theorem}\label{CH}
Here we include a combinatorial proof of the Cayley-Hamilton Theorem which is similar to that of Straubing \cite{Straubing}. We let $A$ denote the $n \times n$ matrix whose entry $A_{i,j}$ is the indeterminate $a_{i,j}$. 
\begin{definition}

For a $k$-tuple $\lambda= (\lambda_0,\ldots, \lambda_{k})$, we let $C(\lambda)$ be the sub-string 
\[
C(\lambda) = (\lambda_{l_1+1}, \ldots, \lambda_{l_2})
\] 
such that $\lambda_{l_1} = \lambda_{l_2}, \lambda_{i} \neq \lambda_{j}$ for $l_1<i< j$, if such a string exists. We say that $(l_1,l_2)$ are the first-rep indices of $\lambda$. 
\end{definition}

 

\begin{theorem}\label{t CH}

Define elements $J_i \in R$ by 
\[
\det(I-tA)= \sum_{i=0}^n J_i t^i.
\]
Then 
 \begin{equation}\label{t CH diag}
\sum_{i=0}^n J_i (A^{n-i})_{r,r}=0, 
 \end{equation}
 
and for $r \neq l$
 \begin{equation}\label{t CH off}
\sum_{i=0}^{n-1} J_i (A^{n-i})_{r,l}=0 .
 \end{equation}
\end{theorem}
\begin{proof}
In the expansion of $(A^{n-i})_{u,v}$ we index each monomial by an $(n-i+1)$-tuple $\lambda$ of integers in $[1,n]$ 
\[
\lambda = (\lambda_0, \ldots, \lambda_{n-i})
\]
such that $\lambda_0=u$ and $\lambda_{n-i} = v$. 

Each $J_i$ is a sum of terms indexed by a pair $(S, \sigma)$ where $S = (s_1, s_2, \ldots, s_i)$ is a subset of $[1,n]$, with $s_j < s_{j+1}$, and $\sigma$ is a permutation of $S$. Thus 
\[
J_i = \sum_{(S,\sigma)} (-1)^{|\sigma|} \prod_{j=1}^i a_{s_j,\sigma(s_j)}
\]
where $|\sigma|$ is the number of cycles of $\sigma$: the permutation $\sigma$ has signature 
\[
\prod_{c \in \sigma}(-1)^{1+\text{length of } c} = (-1)^{i + \#\text{cycles of } \sigma}
\]
where the product is over the cycles of $\sigma$, and each factor of $t$ contributes a factor of $(-1)$. 

We expand out each term in the sums \eqref{t CH diag} or \eqref{t CH off} into a sum monomials in $R$ and construct a sign-reversing involution $I$ on those monomials. 

To a monomial in the $i$-th term we give the index 
\[
(\lambda, S, \sigma)
\]
where $|S|=i$, $\lambda$ is empty if $i=n$, and $S$ and $\sigma$ are empty if $i=0$.

Suppose $i=n$. This may only occur for the sum \eqref{t CH diag}. Then $\lambda$ is empty. Take the cycle $c$ of $\sigma$ that contains $r$: 
\[
c = (c_1, c_2, \dots c_m)
\] 
where $c_1=r$. Let $\lambda'$ be determined by $\lambda_j = c_{j+1}$ for $0\leq j \leq m-1$ and $\lambda_m = r$. Let $\sigma'$ be the permutation that results by removing the cycle $c$ from $\sigma$ and let $S'$ be the resulting set. Then $I$ maps
\[
(\lambda, S, \sigma) \mapsto (\lambda', S', \sigma').
\]

Now suppose $i \neq n$, so $\lambda$ is not empty. Let $h$ be the greatest integer such that $\lambda_{h} \in S$. 

If no such $h$ exists, then 
we claim $C(\lambda)$ exists. If $\lambda_0 = \lambda_{n-i}$, then $C(\lambda)$ exists because we have at least one pair of indices  $l_1,l_2$ where $\lambda_{l_1} = \lambda_{l_2}$, namely $l_1=0$ and $l_2 = n-i$. If $\lambda_0 \neq \lambda_{n-i}$ and $C(\lambda)$ does not exist, then the $n-i+1$ integers in $\lambda$ are distinct, and since $h$ does not exist by assumption, this set and $S$ (which consists of $i$ distinct integers by construction) are disjoint. This contradicts the fact that the integers are in $[1,n]$. This proves the claim. Adjoin this cycle $C(\lambda)$ to $\sigma$ to obtain a permutation $\sigma'$ on the resulting set $S'
$, and remove $C(\vect{\lambda})$ from $\vect{\lambda}$ to obtain a labeling $\lambda'$ on $P_{j}$. Then $I$ maps
\[
( \lambda, S, \sigma) \mapsto ( \lambda', S', \sigma').
\]

Now suppose such an $h$ does exist and consider the cycle $c$ of $\sigma$
\[
c = (c_1, \ldots, c_m)
\]
where $c_1 = \lambda_h$. If $C(\lambda)$ exists, then let $(l_1, l_2)$ be the first-rep indices. If $C(\lambda)$ exists and $h>l_1$, or if $C(\lambda)$ does not exist, then let $\lambda'$ be 
\[
\lambda' = (\lambda_0, \ldots, \lambda_h, c_2, \ldots, c_m, \lambda_h, \lambda_{h+1}, \ldots, \lambda_{n-i}),
\]
and remove $c$ from $\sigma$ to obtain $\sigma'$ and the resulting set $S'$. Then $I$ maps
\[
(\lambda, S, \sigma) \mapsto ( \lambda', S', \sigma').
\]

Otherwise assume $h<l_1$ ($h$ cannot be equal to $l_1$, for then it would also be equal to $l_2$, contradicting the assumption on the maximality of $h$).
Remove the substring $C(\lambda)$ from $\lambda$ to obtain $\lambda'$ and adjoin $C(\lambda)$ as a cycle to $\sigma$ to obtain $\sigma'$ and the resulting set $S'$. Then $I$ maps
\[
( \lambda, S, \sigma) \mapsto ( \lambda', S', \sigma').
\]

It is straightforward to check that $I$ is an involution.
This completes the proof. 
\end{proof}

\section{Gr\"obner basis computations}\label{groebner}

We consider other elements of the form 
\[
z(\mathrm{fern}_{d,n}, \mu)
\]
for certain $d, n$ and $\mu$. We have 
From Theorem \ref{t membership}
\[
z(\mathrm{fern}_{1,n}, \mu) \in J_{1,n}
\]
for all $n$ and $\mu$. (Here $\mu$ is labeling of the root and the one leaf vertex.) For the following $d$ and $n$ we have used Gr\"obner bases implemented by Mathematica \cite{Wolfram} and Singular \cite{Decker} to determine membership in $J_{d,n}$ or $\sqrt{J_{d,n}}$.

For $n=2$, we have 
\[
z(\mathrm{fern}_{d,2}, \mu) \in J_{1,2}\, \, \, d=2,3
\]
for all root-leaf labelings $\mu$. 

We denote a root-leaf labeling of $\mathrm{fern}_{d,n}$ by 
\[
(r,t_0, \ldots, t_{n-2}, t_{n-1})
\]
where $r$ is the label of the root, $t_i$ is a $(d-1)$-tuple of the labels of the vertex $v_i$ on the leftmost path for $0 \leq i \leq n-2$, and $t_{n-1}$ is a $d$-tuple of the labels of the next-to-last vertex $v_{n-1}$ on the leftmost path. 

For $n=3$ and $2\leq d\leq 4$ we have 
\[
z(\mathrm{fern}_{d,3}, \mu) \in J_{1,3}
\] 
for all for all root-leaf labelings $\mu$ except 
\begin{align*}
(1,(2), (3), (1,1)) \text{ and }  (1,(3), (2), (1,1))  \notin J_{2,3} \\ 
 (1,(2,2), (3,3,), (1,1,1)) \text{ and } (1,(3,3), (2,2), (1,1,1))  \notin J_{3,3}\\ 
  (1,(2,2,2), (3,3,,3), (1,1,1,1)) \text{ and } (1,(3,3,3), (2,2,2), (1,1,1,1))  \notin J_{4,3}
\end{align*}
However we do have 
\begin{align*}
z(\mathrm{fern}_{d,3}, (1 ,(2), (3), (1,1)))+  z(\mathrm{fern}_{d,3}, (1 ,(3), (2), (1,1))) \in  J_{1,3}\\ 
z(\mathrm{fern}_{d,3}, (1 ,(2,2), (3,3), (1,1,1)))+  z(\mathrm{fern}_{d,3}, (1 ,(3,3), (2,2), (1,1,1))) \in  J_{3,3}\\ 
z(\mathrm{fern}_{d,3}, (1 ,(2,2,2), (3,3,3), (1,1,1,1)))+  z(\mathrm{fern}_{d,3}, (1 ,(3,3,3), (2,2,2), (1,1,1,1))) \in  J_{4,3}
\end{align*}
and 
\begin{align*}
z(\mathrm{fern}_{d,3}, (1 ,(2), (3), (1,1)))^2 \in  J_{1,3}\\ 
z(\mathrm{fern}_{d,3}, (1 ,(2,2), (3,3), (1,1,1)))^2 \in  J_{3,3}\\ 
z(\mathrm{fern}_{d,3}, (1 ,(2,2,2), (3,3,3), (1,1,1,1)))^2 \in  J_{4,3}
\end{align*}

If we assume that the matrix $a$ satisfies $A^2 = 0$, then let $I_{\mathrm{nil-2}}$ denote the ideal of $R$ generated by elements of the form
\[
(A^2)_{i,j}
\]
for $1\leq i,j\leq n$. Then 
\begin{align*}
z(\mathrm{fern}_{d,3}, (1 ,(2), (3), (1,1)))  \in J_{2,3} \cup I_{\mathrm{nil-2}} \\ 
\end{align*}

Let $I_{\mathrm{char}}$ denote the non-constant coefficients of the characteristic polynomial of $A$. Then 
\begin{align*}
z(\mathrm{fern}_{d,3}, (1 ,(2), (3), (1,1)))  \in J_{2,3} \cup I_{\mathrm{char}} \\ 
\end{align*}

\section{Further work}\label{Further work}
\begin{itemize}

\item See if ideal membership results can be obtained for $n=2$ and all $d$. 

\item Consider the maps $f(x) = x+H$ where the component functions of $H$ are linear forms raised to different degrees.

\item  See if there are corresponding results for the different differential operators of the General Vanishing Conjecture. 

\item Consider ideal membership questions using the Jacobian condition along with the nilpotent assumptions of Dru\.zkowski. 

\item For a matrix $A$, see how to express the coefficients of the characteristic polynomial or their powers in terms of the entries of $A^k$. 
 
\end{itemize}

\end{document}